\documentclass{article}
\usepackage{amsfonts}
\usepackage{amsmath}
\usepackage{amssymb}
\usepackage{graphicx}
\usepackage[english]{babel}
\usepackage{amsthm}
\usepackage{color}

\usepackage{pstricks}
\usepackage{pstricks-add}
\usepackage{pst-xkey}
\usepackage{pst-math}

\newtheorem*{rep@theorem}{\rep@title}
\newcommand{\newreptheorem}[2]{%
\newenvironment{rep#1}[1]{%
 \def\rep@title{#2 \ref{##1}}%
 \begin{rep@theorem}}%
 {\end{rep@theorem}}}
\makeatother

\newtheorem{theorem}{Theorem}[section]
\newtheorem*{theoremp}{Theorem}
\newreptheorem{theorem}{Theorem}
\newtheorem{lemma}[theorem]{Lemma}

\newtheorem{conjecture}[theorem]{Conjecture}

\newtheorem*{claim}{Claim}
\newcommand{\R}{\mathbb{R}}

\newcommand{\F}{\mathcal{F}}

\DeclareMathOperator{\conv}{conv}

\title{On a problem by Dol'nikov}
\author{{\sc Jes\'us Jer\'onimo-Castro\footnote{Supported by CONACYT, SNI
38848}}\\[-1.5mm]
 {\footnotesize Facultad de Ingenier\'ia,}\\[-1.5mm]
  {\footnotesize Universidad Aut\'onoma de Quer\'etaro}\\[-1.5mm]
  {\footnotesize Cerro de las Campanas s/n, C.P. 76010}\\[-1.5mm]
{\footnotesize Quer\'etaro, M\'exico}\\[-1.5mm]
   {\footnotesize {\tt jesusjero@hotmail.com}} \and
   {\sc Alexander Magazinov}\footnote{Partially supported by ERC Advanced Research Grant no 267165 (DISCONV) and RFBR grant 13-01-00563} \\ Steklov Mathematical Institute \\[-1.5mm]
    {\footnotesize Russian Academy of Sciences} \\[-1.5mm]
     {\footnotesize 8 Gubkina street, Moscow 119991} \\[-1.5mm]
      {\footnotesize Russia} \\[-1.5mm]
       {\footnotesize \texttt{magazinov-al@yandex.ru}} \and
        {\sc Pablo Sober\'on}\footnote{Partially supported by ERC Advanced Research Grant no 267165 (DISCONV)} \\[-1.5mm]
 {\footnotesize Department of Mathematics} \\[-1.5mm]
 {\footnotesize University of Michigan} \\[-1.5mm]
 {\footnotesize 530 Church Street, Ann Arbor} \\[-1.5mm]
 {\footnotesize MI 48109-1043, US} \\[-1.5mm]
 {\footnotesize \texttt{psoberon@umich.edu}} }

\begin{document}

\maketitle

\abstract{In 2011 at an Oberwolfach workshop in Discrete Geometry,
V. Dol'nikov posed the following problem. Consider three non-empty
families of translates of a convex compact set $K$ in the plane.
Suppose that every two translates from different families have a
point of intersection. Is it always true that one of the families
can be pierced by a set of three points?

A result by R. N. Karasev from 2000 gives, in fact, an affirmative
answer to the ``monochromatic'' version of the problem above. That
is, if all the three families in the problem coincide. In the
present paper we solve Dol'nikov's problem positively if $K$ is
either centrally symmetric or a triangle, and show that the
conclusion can be strengthened if $K$ is an euclidean disk.  We also confirm the conjecture if we are given four families satisfying the conditions above.}

\section{Introduction}

Helly's celebrated theorem \cite{Radon:1921vh, Helly:1923mengen}
gives a characterization of all families of convex sets in $\R^d$
that have a point of intersection.  Namely, it says that \textit{A
finite family $\mathcal{F}$ of convex sets in $\R^d$ has a point
of intersection if and only if every $d+1$ sets of $\mathcal{F}$
have a point of intersection.}  This theorem is optimal in the
sense that the number $d+1$ cannot be improved.  The question then
becomes whether weakening Helly's condition can still give results
on the global intersection structure of the family of convex sets.
This was answered positively by Alon and Kleitman in 1992
\cite{Alon:1992ta}.

We say that a family $\F$ of convex sets in $\R^d$ has
\textit{piercing number} $k$ if $k$ is the smallest positive
integer such that there is a set of $k$ points in $\R^d$ that
intersects every element of $\F$.  We denote this by $\pi (\F) =
k$.  We say that a family  $\F$ of convex sets in $\R^d$ has the
$(p,q)$ property if out of every $p$ sets in $\F$ we can always
find at least $q$ which are intersecting.   Alon and Kleitman gave
a positive answer to the $(p,q)$ conjecture of Hadwiger and
Gr\"unbaum, showing that \textit{for every $p \ge d \ge d+1$ there
is a constant $c = c(p,q,d)$ such that for every family $\F$ of
convex sets in $\R^d$ with the $(p,q)$ property, we have $\pi (
\F) \le c$}.  However, the current bounds on $c$ are astronomical.  This is commonly known as the $(p,q)$ theorem.

Finding precise bounds for $c$ is still an open problem.  Even in
the first non-trivial case, the conjecture is $c(4,3,2) = 3$ but
the best bound so far gives $c(4,3,2) \le 13$
\cite{Kleitman:2001vo} (the existence theorem gives a bound of
over $200$).

The condition $q \ge d+1$ is essential, as a family of
$n$ hyperplanes in general position in $\R^d$ has the $(d,d)$
property but cannot be pierced by less than $\frac{n}{d}$ points.
However, if further conditions are imposed on the family, $(p,q)$
conditions with $q \le d$ can give bounds on the piercing number
of the family.  This can bee seen in the following result of Karasev.

\begin{theoremp}[Karasev, 2000 \cite{Karasev:2000ui}]
Let $K$ be a convex set in the plane.  If $\mathcal{F}$ is a
family of pairwise intersecting translates of $K$, then
$\pi (F) \le 3$.
\end{theoremp}

In general, if $\F$ is a family of translates of a convex body $K$
in $\R^d$, a $(p,2)$ property is enough to bound the piercing
number.  This was first noted by Kim, Nakprasit, Pelsmajer and
Skokan \cite{Kim:2006gl} with a bound of $\pi(\F) \le 2^d d^d
(p-1)$ and recently an improved bound of $\pi (\F) \le 2^d
\binom{2d}{d}(d\log d + \log \log d + 5d)(p-1)$ was obtained by
Nasz\'odi and Taschuk \cite{Naszodi:2010km}, also showing that the
order of growth of their bound is correct.

More recently, a result by Katchalski and Nashtir \cite{KN11} extends Karasev's
result to more diverse families of convex sets, rather than just
translates of the same body.  We say that two polygons $P,
Q$ are \emph{related} if $P$ is the intersection of $m$
half-planes $a_1, a_2, \ldots, a_k$ and $Q$ is the intersection
of $m$ half-planes $b_1, b_2, \ldots, b_k$, so that each $b_i$ is
a translate of $a_i$.

\begin{theoremp}[Katchalski and Nashtir, 2011]
If $\F$ is a family of pairwise intersecting convex sets in the
plane each of which is related to a fixed $n$-gon, then $\pi(\F)\le 3^{\binom{n}{3}}$.
\end{theoremp}

One should note that, for the case of triangles, the result above gives the
precise bound for Karasev's result.  In the case of parallelograms it also shows that a $(2,2)$ property implies a piercing number of one.

We wish to explore the relation of the results above with what is commonly known as
\emph{colourful} theorems.  The classical example of this kind of
results is a notable variation of Helly's theorem, also known as
the colourful version of Helly's theorem

\begin{theoremp}[Lov\'asz, 1982 \cite{Barany:1982va}]
Let $\mathcal{F}_1, \mathcal{F}_2, \ldots, \mathcal{F}_{d+1}$ be
families of convex sets in $\mathbb{R}^d$.  If every $(d+1)$-tuple
$K_1 \in \mathcal{F}_1, K_2 \in \mathcal{F}_2, \ldots, K_{d+1} \in
\mathcal{F}_{d+1}$ has a point of intersection, then there is an
$i_0 \in \{ 1,2, \ldots, d+1\}$ such that all the sets in
$\mathcal{F}_{i_0}$ have a point of intersection.
\end{theoremp}

Lov\'asz's proof for the theorem above appeared first in a paper
by B\'ar\'any where the colourful version of Carath\'eodory's
theorem is presented.  Note that if all $\mathcal{F}_i$ are equal,
we get the original statement of Helly's theorem.

Vladimir Dol'nikov asked whether Karasev's result could have a
coloured version in the same spirit \cite{Dol}. Namely, he posed
the following conjecture.

\begin{conjecture}[Dol'nikov's problem]\label{conj:dolnikov}
Let $K$ be a compact convex set in the plane and $\F_1, \F_2,
\F_3$ are three non-empty finite families of translates of $K$.
Suppose that $K' \cap K'' \neq \varnothing$ for every pair $(K',
K'')$ such that $K' \in \F_i, K'' \in \F_j$ and $i \neq j$. Then
there exists $m \in \{1,2,3\}$ such that $\pi (\F_m) \le 3$.
\end{conjecture}

In other words, consider a finite family
of translates of $K$ coloured with three colours such that
there is at least one translate of each colour. Assume, in addition, that
every two translates of different colour have a non-empty
intersection. Then we can choose a colour such that all translates
of this colour can be pierced by three points.

The purpose of this paper is to confirm
Conjecture~\ref{conj:dolnikov} in two special cases: if $K$ is
centrally symmetric or if $K$ is a triangle.  We also show a much
stronger statement if $K$ is an euclidean disk. That is, the main results
of this paper are the following theorems.

\begin{theorem}\label{dolnikov-symmetric}
Let $K$ be a centrally symmetric compact convex set in the plane
and $\F_1, \F_2, \F_3$ are three non-empty finite families of
translates of $K$.  Suppose that $K' \cap K'' \neq \varnothing$
for every pair $(K', K'')$ such that $K' \in \F_i, K'' \in \F_j$
and $i \neq j$. Then there exists $m \in \{1,2,3\}$ such that $\pi
(\F_m) \le 3$.
\end{theorem}

\begin{theorem}\label{main-triangles}
Let $T$ be a closed triangle in the plane and $\F_1, \F_2, \F_3$
are three non-empty finite families of translates of $T$.  Suppose
that $K' \cap K'' \neq \varnothing$ for every pair $(K', K'')$
such that $K' \in \F_i, K'' \in \F_j$ and $i \neq j$. Then there
exists $m \in \{1,2,3\}$ such that $\pi (\F_m) \le 3$.
\end{theorem}

\begin{theorem}\label{euclidean-balls}
Let $P_1,\,P_2,\, \ldots ,P_k$ be finite families of euclidean
disks of diameter $1$, with $k\geq 2$. Suppose that $B' \cap B''
\neq \varnothing$ for every pair $(B', B'')$ such that $B' \in
P_i, B'' \in P_j$, $i\neq j$.  Then there exists $m \in
\{1,2,\ldots, k\}$ such that $\pi (\bigcup_{i \neq m} P_i) \le 3$.
\end{theorem}

For Theorem~\ref{dolnikov-symmetric} we give two different proofs.
The first uses Helly's colourful theorem, and the second is
constructive.

It should be noted that colourful version of the $(p,q)$ theorem
have been studied recently by B\'ar\'any, Fodor, Montejano,
Oliveros and P\'or \cite{BMF13}.  They have sharp results in
dimension one and an existence similar to the $(p,q)$ theorem for
general $d$.  Since they work with general
convex sets, their results require $q \ge d+1$.

We believe that the strong conclusion of Theorem \ref{euclidean-balls} holds for any convex set $K$.

\begin{conjecture}
Let $K$ be a compact convex set in the plane and $\F_1, \F_2,$
$\ldots, \F_k$ are non-empty finite families of translates of $K$.
Suppose that $K' \cap K'' \neq \varnothing$ for every pair $(K',
K'')$ such that $K' \in \F_i, K'' \in \F_j$ and $i \neq j$. Then
there exists $m \in \{1,2,\ldots, k\}$ such that $\pi (\bigcup_{i
\neq m}\F_i) \le 3$.
\end{conjecture}

The paper is structured as follows.  In Section \ref{section:restatement}, we show that the theorems
above can be translated to statements regarding covering families
of points with few translated copies of $-K$ instead of piercing
families of translated copies of $K$ with few points.  

In sections
\ref{section-centrally-symmetric1} and
\ref{section-second-proof} we give two different proof of
Theorem \ref{dolnikov-symmetric}.  In Section
\ref{section:euclidean-disks} we give the proof of Theorem
\ref{euclidean-balls}.  In section
\ref{section:triangles} we give the proof for Theorem
\ref{main-triangles}.  Finally, in section \ref{section:final-remark} we prove Dol'nikov's conjecture if we use four colours instead of three, by following directly the arguments used by Karasev in \cite{Kar08}.

\begin{theorem}\label{thm:four-colors}
Let $K$ be a compact convex set in the plane and $\F_1, \F_2,
\F_3, \F_4$ are four non-empty finite families of translates of $K$.
Suppose that $K' \cap K'' \neq \varnothing$ for every pair $(K',
K'')$ such that $K' \in \F_i, K'' \in \F_j$ and $i \neq j$. Then
there exists $m \in \{1,2,3,4\}$ such that $\pi (\F_m) \le 3$.
\end{theorem}

\section{Restatement of Conjecture~\ref{conj:dolnikov}}\label{section:restatement}

The purpose of this section is to fix notation and give an equivalent version of theorems \ref{dolnikov-symmetric}, \ref{main-triangles} and \ref{euclidean-balls} which are easier to analyse.

Given a family of translates of a convex set $K$, statements about its piercing number have a natural translation to statements regarding a family of points being covered with few copies of $-K$.  This is a standard technique, see for instance \cite{Karasev:2000ui}.

We should also mention that
lemmas~\ref{piercing-covering}~and~\ref{piercing-distance} below can be
directly generalised to the case of $\R^d$.

\begin{lemma}\label{piercing-covering}
Let $K$ be a planar compact convex set. Let
$$\F = \{K + t_1, K + t_2, \ldots, K + t_n \}$$
be a family of translates of $K$. Suppose that $\pi(\F) = p$. Then
the points $t_1, t_2, \ldots, t_n$ can be covered by $p$
translates of $-K$.
\end{lemma}

\begin{proof}
Assume that the points $x_1, x_2, \ldots, x_p$ pierce the family
$\F$. For an arbitrary $1\le i\le n$ the translate $K + t_i$ is
pierced by some $x_j$. Then there is a point $a \in K$ such that $a + t_i = x_j.$

Therefore $x_j - t_i \in K$. This immediately implies $t_i - x_j
\in - K$ and, in particular, $t_i \in (- K + x_j)$.  Thus, every $t_i$ is covered by at
least one of the $p$ translates $- K + x_1, - K + x_2, \ldots, - K + x_p$.
\end{proof}

\begin{lemma}\label{piercing-distance}
Let $K$ be a planar compact convex set. Suppose that $K + t_i$ ($i
= 1,2$) are two translates of $K$ and
$$(K + t_1) \cap (K + t_2) \neq \varnothing.$$
Let $\rho(\cdot, \cdot)$ be the Minkowski metric in the plane with
the Minkowski difference body $\frac{1}{2}D(K)=\frac{K + (-K)}{2}$
as the unit ball. Then $\rho(t_1, t_2) \leq 2$.
\end{lemma}

\begin{proof}
Let $x \in (K + t_1) \cap (K + t_2)$. Since $x \in (K + t_1)$,
there exists a point $a_1 \in K$ such that $x = a_1 + t_1$.
Similarly, $x = a_2 + t_2$ for some $a_2 \in K$. Therefore
$$ t_1 - t_2 = a_2 - a_1, $$
thus the point $t_1 - t_2$ belongs to the Minkowski sum $K + (-
K)$, which is the $\rho$-ball of radius $2$.

\end{proof}

The two lemmas above allow to restate
Conjecture~\ref{conj:dolnikov} as follows.

\begin{conjecture}\label{conj:restated}
Let $X_1$, $X_2$, $X_3$ be three non-empty finite point sets in
the plane, $K$ be a compact convex planar set and
$\rho(\cdot, \cdot)$ be the Minkowski metric in the plane with
unit ball $\frac{1}{2}D(K)$. Suppose that $\rho (x', x'') \leq 2$
for every pair of points $(x', x'')$ such that $x' \in X_i, x''
\in X_j$ and $i \neq j$. Then there exists $m \in \{1,2,3\}$ such
that $X_m$ can be covered by three translates of $K$.
\end{conjecture}

Conjecture~\ref{conj:restated} is equivalent to
Conjecture~\ref{conj:dolnikov}. However, the difference is that
the set denoted by $K$ in Conjecture~\ref{conj:dolnikov} plays the
role of $- K$ in Conjecture~\ref{conj:restated}. We also emphasise
that for centrally symmetric $K$ centred at the origin one has
$$K = - K = \frac{1}{2}D(K).$$

Similar statements for theorems \ref{dolnikov-symmetric}, \ref{main-triangles} and \ref{euclidean-balls} follow naturally.

\section{First proof of Theorem \ref{dolnikov-symmetric}}\label{section-centrally-symmetric1}

In this section we prove Conjecture~\ref{conj:restated} under
assumption that $K$ is centrally symmetric. As mentioned above,
$\rho(\cdot, \cdot)$ is the Minkowski metric that has $K$ as unit ball.

The Jung constant or the Jung radius, $J(\rho)$ is defined as the infimum of the set of non-negative real numbers $\mu$ which have the following property:
\textit{given any family of pairwise intersecting translates $\{x_i + K:
i\in I\}$ then $\bigcap _{i\in I}(x_i + \mu K)\neq\emptyset$}. In
other words, for every set of points $X$ such that $\rho(x, y)\leq
2$ for any $x,y\in X,$ we have that $X$ is contained in some
translate of $\mu K.$ A $\rho$-ball of radius $J(\rho)$ is
called a {\it Jung ball}. This term has been defined in \cite{Gru59} and it
is known that $J(\rho)\geq 1$ with equality if and only if $K$
is a parallelogram.  We denote by $B_{J(\rho)} (z)$ the $\rho$-ball of radius $J(\rho)$ centred and $z$.

Conjecture \ref{conj:restated} will follow from the following lemma.

\begin{lemma} \label{co22}
Let $X_1$, $X_2$, $X_3$ be three non-empty finite point sets in
the plane, $K$ be a centrally symmetric compact convex planar set and
$\rho(\cdot, \cdot)$ be the Minkowski metric in the plane with
unit ball $K$. Suppose that $\rho (x', x'') \leq 2$
for every pair of points $(x', x'')$ such that $x' \in X_i, x''
\in X_j$ and $i \neq j$. Then, there exists $m \in
\{1,2,3\}$ such that $X_m$ can be covered with one ball of radius
$J(\rho)$.
\end{lemma}

\begin{proof}
Consider a rainbow triangle $\triangle x_1x_2x_3,$ with $x_1\in X_1,$ $x_2\in X_2,$ $x_3\in X_3.$ Since the $\rho$-diameter of $\triangle x_1x_2x_3$ is at most two, we have that there exists a translate of the Jung ball, $B_{J(\rho)}(z),$ which contains $\triangle x_1x_2x_3$. It is clear that $z\in B_{J(\rho)}(x_1)\cap B_{J(\rho)}(x_2)\cap B_{J(\rho)}(x_3).$ Now, we apply Helly's colourful theorem to the three families of balls 
\[
\{B_{J(\rho)}(x_1) : x_1\in X_1\}, \quad \{B_{J(\rho)}(x_2) : x_2\in X_2\}, \quad \{B_{(\rho)}(x_3) : x_3\in X_3\}.
\]
 Hence at least one of the families of balls has a non-empty intersection. Without loss of generality, assume that $$x\in \bigcap\limits_{x_1\in X_1} B_{J(\rho)}(x_1).$$ 
Which is equivalent to $X_1\subset B_{J(\rho)}(x)$.
\end{proof}

However, it was shown by Gr\"unbaum that we can always cover $B_{J(\rho)}(x)$ by three unit $\rho$-balls \cite[Theorem II]{Gru57}.  The proof consists of finding three points $x_1,x_1,x_3$ in the boundary of $B_{J(p)}$ such that the $\rho$-distance between the extremes of two pairs $x_i, x_j$ is equal to $2$ and the the distance between the last pair is less than $2$.  Taking the three $\rho$-balls with diameters of the form $[x_i,x_j]$ completes the proof.  This confirms Conjecture~\ref{conj:restated} (and then
Conjecture~\ref{conj:dolnikov}) for centrally symmetric sets, i.e.
proves Theorem~\ref{dolnikov-symmetric}.

\section{Second proof of Theorem \ref{dolnikov-symmetric}}
\label{section-second-proof}
In this section we present a constructive proof of Theorem \ref{dolnikov-symmetric}.  As before, let $K$ be a centrally symmetric convex set and $\mathcal{F}_1, \mathcal{F}_2$ and $\mathcal{F}_3$ be families of translates of $K$.

Without loss of generality, we may suppose $K$ is smooth and that $0 \in
K$.  Given $x+K$ a translate of $K$, we call $x$ the
\textit{centre} of $x+K$.  For any point $y$ in the boundary of
$K+x$, we call $y-x$ a \textit{radius} of $K$.  If $y$ and $z$ are
points in the boundary of $x+K$ such that $x$ is in the segment with
endpoints $z,y$, we call $z-y$ a \textit{$K$-diameter} of $x+K$.  Notice that $K$-diameters are precisely those segments with length $2$ according to the Minkowski metric with unit ball $K$.
\\

We may assume without loss of generality that every finite intersection of members of $\mathcal{F}_1 \cup \mathcal{F}_2 \cup \mathcal{F}_3$ has a unique point with minimal $y$-coordinate.  If a convex set has a unique point with minimal $y$ coordinate, we will call it the $y$-directional minimum.  In this section, we will say a point $x_1$ is above (resp. below) $x_2$ if its $y$-coordinate is larger (resp. smaller).

Let $K_1$, $K_2$ be the two
copies of $K$ in different families such that the $y$-directional
minimum $p$ of $K_1 \cap K_2$ is maximal.

Suppose that $K_1 \in \F_1$ and $K_2 \in \F_2$.  We show that
$\pi(F_3) \le 3$.  Let $x_1, x_2$ be the centres of $K_1, K_2$
respectively.  We know that every translate $K_3$ of $K$ that
intersects $K_1 \cap K_2$ must contain $p$.  This is a standard
argument that can be traced back to Helly's own proof of his
theorem \cite{Helly:1923mengen}.  The interested reader may find
further references in \cite{danzer1963helly}.

Let $S_1, S_2$ be
two copies of $2K$ around $x_1, x_2$
respectively.  For $K_3\in \mathcal{F}_3$ to intersect both $K_1$ and $K_2$, its centre
$x_3$ must be in $S_1 \cap S_2$.  Consider $\mathcal{F}^*_3$ the set of all translates in $\F_3$ that do not contain $p$.  We want to show
that $\pi (\F^* _3 ) \le 2$.

Let $K^*$ be the translate of $K$ with centre $p$.  Note that $x_1, x_2
\in \partial K^*$, the boundary of $K$.   Let $y_1, y_2$ be points on $\partial K^*$
such that $x_1 - y_1$ and $x_2 - y_2$ are $K$-diameters. The centres
of the translates of $K$ in $\F^*_3$ must be in $(S_1 \cap
S_2)\backslash (-K^*)$.  Also, since the $y$-directional minimum
of their intersection with $K_1$ and $K_2$ must be lower than $p$,
they must be in the lower component of this set.  Let $A$ be the
resulting set of positions where the centres of $\F^*_3$ can be.
It remains to show that $A$ can be covered by two copies of $K$.
\\

There are two natural candidates for these translates.  Consider $q$
the lower point of intersection of the boundaries of $S_1$ and
$S_2$.  
Note that since $q-x_1$ and $q-x_2$ are radii of $2K$, they are also diameters of $K$.  Let $R_1$ and $R_2$ be the two translates of $K$
that have $q-x_1$ and $q-x_2$ as $K$-diameters, respectively.

\begin{claim}
$A \subset R_1 \cup R_2$
\end{claim}

\begin{figure}
\centerline{
\includegraphics[scale=0.7]{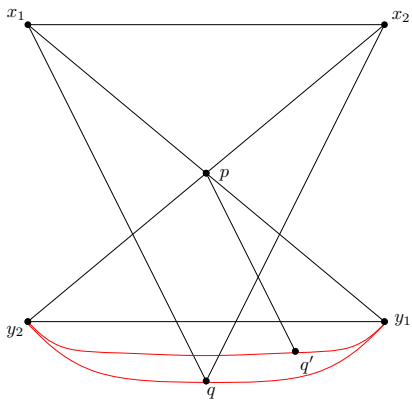}
}
\caption{Arcs showing region $A$}\label{figura1}
\end{figure}

For simplicity we assume that $\triangle x_1 q x_2$ is
oriented as in figure \ref{figura1}.  Let $A_1$ be the section of $A$ left of
the line $q x_1$, $A_2$ the section of $A$ inside triangle $\triangle
x_1 q x_2$ and $A_3$ the section of $A$ right of the line $q x_2$. It
remains to show that $A_1 \subset R_1$, $A_2 \subset R_1 \cup R_2$
and $A_3 \subset R_2$.  The arguments for $A_1$ and $A_3$ are
analogous.
\\

\noindent \textbf{Part 1.} $A_1 \subset R_1$
\\

Let $r$ be the point in the segment $x_1 q$ that is the centre of
$R_1$ (see Figure \ref{figura2} below).  Let $q' = p + (q-r)$, note that $q'$ is in the boundary of
$A$.  The vectors $r-p$ and $q-q'$ are equal.  Let $y_2^* = y_2 +
(r-p) = r+ (y_2 - p)$ and $x_1^* = x_1 + (r-p)$.  If we translate the arc of
boundary of $A$ between $q'$ and $y_2$ by the vector $r-p$ we
obtain an arc in the boundary of $R_1$.  Also $y_2^*$ and $x_1^*$
are in the boundary of $R_1$.
\\

Since $x_1^*$ is in the boundary of $R_1$, it cannot be contained
in $\triangle x_1 y_2^* r$, so $y_2$ cannot be to the left of line
$x_1 y_2^*$ (oriented from $y_2^*$ to $x_1$).  This is because
$x_1-y_2$ is equal to $x_1^* - y_2^*$.  Since the arc of boundary
of $R_1$ between $q$ and $y_2^*$ is a translated copy of the arc
of boundary of $A$ (and $-K^*$) between $y_2$ and $q'$, $y_2$
cannot be to the left of it (oriented from $q$ to $y_2^*$).  Thus
$y_2 \in R_1$.  Also note that $y_2$ and $x_1$ are contained in
$S_2$, so the slope of the boundary of $A$ corresponding to the
arc determined by $S_2$ is bounded by the direction $x_1 y_2$.
This implies that $A_1$ is on the right side of the line $x_1
y_2^*$ and the arc of $R_1$ between $y_2^*$ and $q$, which implies
$A_1 \subset R_1$.
\\
\begin{figure}[htc]
\centerline{
\includegraphics[scale=0.7]{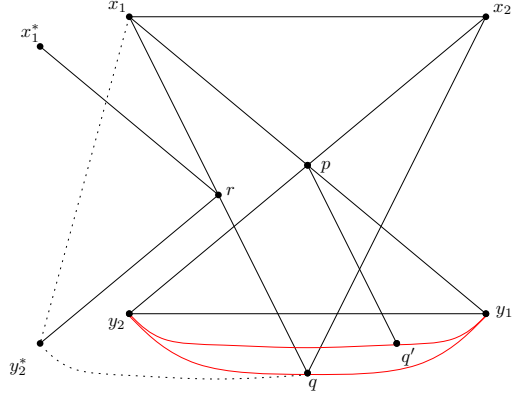}
}
\caption{Arguments for part $1$.}\label{figura2}

\end{figure}

\noindent \textbf{Part 2.} $A_2 \subset R_1 \cup R_2$.
\\
Consider $m_1$ the midpoint of $x_1 q$, $m_2$ the midpoint of $x_2
q$ and $z$ the midpoint of $x_1 x_2$.  The length $|x_1 x_2|$ is
bounded above by the length of the diameter parallel to it in $K$.
Thus $|m_1 m_2| = \frac{|x_1 x_2|}{2}$ and is less than the radius
of $K$ in that direction.  Also $|m_1 z| = |m_2 x_2|$ and is the
length of the radius of $K$ in that direction.  Thus $z, m_2 \in
R_1$ and analogously $z, m_1 \in R_2$.  This shows $A_2
\subset \triangle x_1x_2 q \subset R_1 \cup R_2$.
\qed

\begin{figure}[htc]
\centerline{
\includegraphics[scale=0.7]{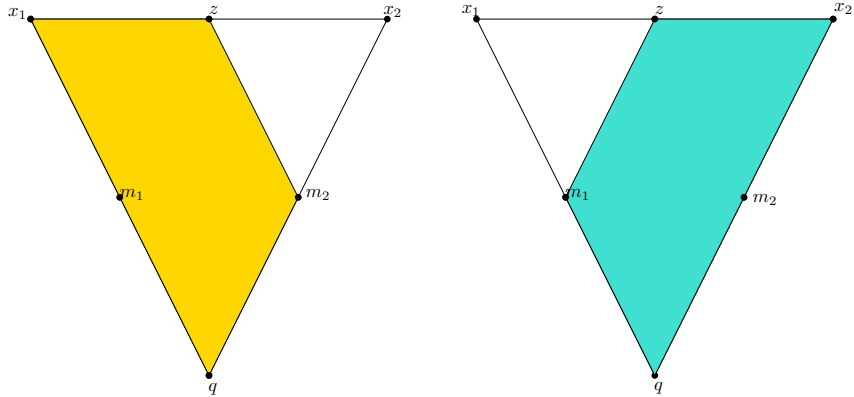}
}
\caption{Arguments for part $2$.  The left and right trapezoids are parts of $\triangle x_1 x_2 q$ covered by $R_1$ and $R_2$, respectively}\label{figure-trapezoids}

\end{figure}

\section{Results for Euclidean disks}\label{section:euclidean-disks}

In this section we prove the following stronger version of Theorem
\ref{euclidean-balls}.  We denote by $[k]$ the set $\{1,2, \ldots, k\}$.

\begin{reptheorem}{euclidean-balls}[Second version]
Let $P_1,\,P_2,\, \ldots ,P_k$ be compact sets of points in the
Euclidean plane, with $k\geq 2$, such that the distance between
any two points in different sets is at most 1. Then, for some
$i\in[k]$, $\bigcup_{j\neq i}P_j$ can be covered by the union of
three balls of diameter less than $1$.
\end{reptheorem}

We use the following notation: for every $x\in
\mathbb{R}^2$ we denote by $B_x$ the unit ball centred at $x$, and $o$ the origin.  Consider
$A(r)=B_{a}\cap B_{b},$ with the points $a,b$ on the $x$-axis
and such that $d(a,o)=d(o,b)=r/2.$

\begin{proof} Given a horizontal line $xy$ we denote by $\Gamma
^{xy}$ the upper closed half-space bounded by $xy,$ and by
$\Gamma_{xy}$ the lower one.  Without loss of generality, we may assume that
$\text{diam}\, P_1 \geq\text{diam}\, P_j,$ for every $j\in\{2,3,\ldots ,k\}.$ Let
$x_1,y_1\in P_1$ such that $r=d(x_1,y_1)=\text{diam}\, P_1$ and
suppose that $x_1,y_1$ are on the $x$-axis and the segment
$[x_1,y_1]$ is centred at the origin $o$. The proof is divided into
two main cases:

\textbf{(1)} $d> 1$. Here we have two subcases:

\begin{enumerate} \item [(1a)] $\frac{2}{\sqrt{3}}\leq d\leq 2$. The case
$d=2$ follows immediately from the conditions of the problem, so
we may assume that $d< 2$.  First, we have that $\bigcup
_{i=2}^kP_i\subset A(r)$.  In this case we show that $A(r)$
can be covered by $3$ balls of diameters strictly smaller than
$1$.  Since for every pair of positive numbers $r_2,r_1$, with
$r_2>r_1$ we have that $A(r_2)\subset A(r_1),$ it is sufficient to
prove the case when $r=\frac{2}{\sqrt{3}}.$ In this case we
proceed as follows: let $a,b$ be the points of intersection
between the boundaries of $B_{x_1}$ and $B_{y_1},$ also let $c$
and $e$ be the points on the boundary of $B(y_1)$ and $B(x_1),$
respectively, such that $\measuredangle cba= \measuredangle
abe=25^{\circ}.$ After some simple calculations we have that
$d(b,e)<1$ and the radius of the circumscribed circle of triangle
$\triangle ace$ is smaller than $1/2$. Then, the circle
circumscribed to $\triangle ace$ and the circles with radii
$\frac{1}{2} d(b,e)$ and centres at the midpoints of $[c,b]$ and
$[b,e],$ respectively, cover $B_{x_1}\cap B_{y_1}.$ Hence,
$A(r)=B_{x_1}\cap B_{y_1}$ can be covered by the union of three
balls of diameter less than $1$ (see Figure \ref{figure:disks1}).
\begin{figure}[h!]
\centerline{ \psset{unit=.8cm}
\begin{pspicture}(0,-.6)(4.5,8.8)
\pscircle[linecolor=blue](2.48,5.17){2.02}
\pscircle[linecolor=blue](1.47,2.25){2.02}
\pscircle[linecolor=blue](3.38,2.25){2.02}
\psarc[linewidth=.04](.2,3.84){4.02}{-75}{75}
\psarc[linewidth=.04](4.8,3.84){4.1}{105}{255}
\psdots[dotsize=3pt](.2,3.84)(4.8,3.84)(4.22,4.11)(.73,4.12)(2.43,7.2)(2.43,.5)(2.48,5.17)(1.47,2.25)(3.38,2.25)
\psline(.2,3.84)(4.8,3.84) \rput{0}(-.2,3.85){\small $x_1$}
\rput{0}(5.2,3.85){\small $y_1$}\rput(2.4,7.5){\small
$a$}\rput(2.4,.2){\small $b$}\rput(.5,4.2){\small
$c$}\rput(4.45,4.12){\small $e$}
\end{pspicture}}
\caption{Disks for case (1a)}
\label{figure:disks1}
\end{figure}
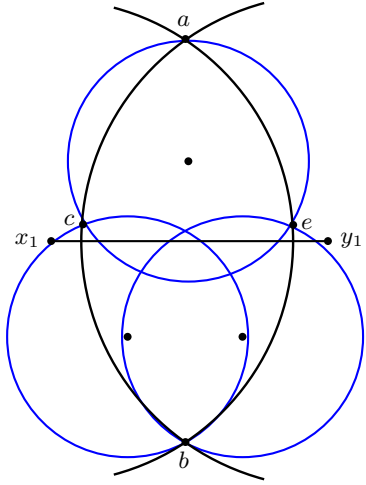

\item[(1b)] $1< d<\frac{2}{\sqrt{3}}$. In this case we show
that $\bigcup _{i=2}^kP_i$ is contained in a special subset of
$A(r)$, which in turn can be covered by $3$ balls of diameters
strictly smaller than $1$. Since for every $r>1$ we have that
$A(r)\subset A(1),$ we know that $\bigcup _{i=2}^kP_i \subset
A(1)$. Set $x_1=(-1/2,0)$, $y_1=(1/2,0)$, and $p=(0,1/\sqrt 3).$
Let $q$ and $s$ be the points where the circle with radius 1 and
centre at $-p$ intersects the boundaries of $B_{x_1}$ and
$B_{y_1}$, (as shown in Figure \ref{figure:disks2}) and let $z,w$ and $x,y$ be the
points where the lines parallel to the $x$-axis, through $p$ and
$-p$, respectively, intersects the boundary of $B_{x_1}\cap
B_{y_1}.$ Now, let $m$ be the midpoint of the segment $[q,a].$ Let
$o_1$ be the centre of the circle circumscribed to triangle
$\triangle aqs. $ We know that $o_1$ is the intersection point
between the segments $[a,o]$ and $[m,y_1].$ Let $\alpha$ be the
value of the angle $\measuredangle qao,$ then standard computations show that $\alpha< 39.3^{\circ}$ and $d(q,a)<0.72$, it follows that
$$d(a,o_1)<\frac{0.36}{\cos 40^{\circ}}<0.47,$$ that is, the radius
of the circle circumscribed to $\triangle aqs$ is smaller than
$1/2$.
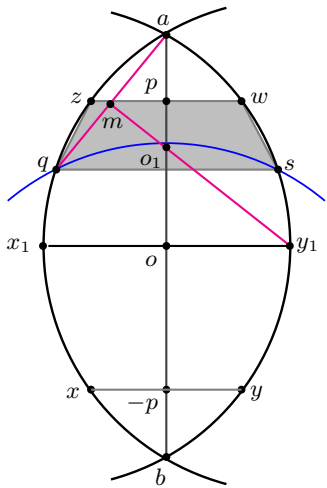
\begin{figure}[h!]
\centerline{ \psset{unit=.8cm}
\begin{pspicture}(0,-.5)(5,8.4)
\psline[linecolor=gray,fillcolor=lightgray,fillstyle=solid](1.2,6.25)(3.7,6.25)(4.3,5.11)(.62,5.11)(1.2,6.25)
\psarc[linewidth=.04](.49,3.84){4.02}{-75}{75}
\psarc[linewidth=.04](4.51,3.84){4.1}{105}{255}
\psarc[linewidth=.03,linecolor=blue](2.45,1.45){4.1}{50}{130}
\psline[linecolor=darkgray](2.45,7.35)(2.45,.33)
\psline[linecolor=magenta](2.45,7.35)(.62,5.11)
\psline[linecolor=magenta](4.51,3.84)(1.52,6.2)
\psdots[dotsize=3pt](4.5,3.84)(.4,3.84)(2.45,3.84)(1.2,6.25)(3.7,6.25)(1.2,1.45)(3.7,1.45)(2.45,1.45)
(2.45,6.25)(2.45,7.35)(2.45,.33)(.62,5.11)(4.31,5.11)(1.52,6.2)(2.45,5.48)
\psline(.49,3.84)(4.51,3.84)
\psline[linecolor=gray](1.2,1.45)(3.7,1.45)
\rput{0}(.95,6.3){\small $z$} \rput{0}(4,6.3){\small $w$}
\rput{0}(0,3.85){\small $x_1$} \rput{0}(4.8,3.85){\small $y_1$}
\rput{0}(.9,1.4){\small $x$} \rput{0}(3.95,1.4){\small
$y$}\rput(2.05,1.2){\small $-p$} \rput(2.35,0){\small
$b$}\rput(2.4,7.6){\small $a$}\rput(2.2,6.5){\small
$p$}\rput(.4,5.2){\small $q$}\rput(4.5,5.2){\small $s$}
\rput(2.2,3.6){\small $o$}\rput(1.56,5.9){\small $m$}
\rput(2.2,5.25){\small $o_1$}
\end{pspicture}}
\caption{Regions for case (1b)}
\label{figure:disks2}
\end{figure}

Now, consider the circle with radius $0.49$ and centre at $o_1$ and
let $q'$ and $s'$ be the points where this circle intersects the
boundary of $B_{x_1}\cap B_{y_1}.$ We have that
$d(-p,q')=d(-p,s')<1$ and the union of the circles with diameters
$[-p,q'],$ $[-p,s'],$ and the circle with radius $0.49$ and centre
at $o_1$ cover the region $B_{x_1}\cap B_{y_1}\cap \Gamma^{xy}.$
Finally, if there is a point of $\bigcup _{i=2}^kP_i$ in the
region $B_{x_1}\cap B_{y_1}\cap \Gamma^{zw}$, then $\bigcup
_{i=2}^kP_i$ must be contained in the region $B_{x_1}\cap
B_{y_1}\cap \Gamma^{xy},$ else, if there is a point in the region
$B_{x_1}\cap B_{y_1}\cap \Gamma_{xy}$ then $\bigcup _{i=2}^kP_i$
must be contained in the region $B_{x_1}\cap B_{y_1}\cap
\Gamma_{zw}.$ In both cases we have proved in the argument given
above that we can cover $\bigcup _{i=2}^kP_i$ with the union of
three balls with diameter strictly less than 1. This concludes the
proof of (1b).
\end{enumerate}

\textbf{(2)} $d\leq 1.$ In this case we have that $\text{diam}\,
\bigcup\limits _{i=1}^kP_i \leq 1$.  It is known that any closed
compact planar set in the plane with diameter $1$ can be covered
with the union of three balls of diameter less than $1$
\cite{Gru57}, which gives us the result we wanted.
\end{proof}

\textbf{Remark 1.} If $k\geq 3$ then, by Lemma ~\ref{co22}, we
have that at least one of the sets $P_i$ is contained in a ball of
diameter $\frac{2}{\sqrt{3}},$ and hence it has diameter at most
$\frac{2}{\sqrt{3}}.$ However, when $k=2$, this is not true.
Indeed, we only know that one of the sets has diameter at most
$\sqrt{2}$: Consider the vertices of a square of side $1,$ $P_1$
as one pair of diagonal points, and $P_2$ as the other pair of
diagonal points. The distance between any point from $P_1$ and any
point from $P_2$ is exactly $1,$ however, we have that
$\text{diam}\,P_1=\text{diam}\,P_2=\sqrt{2}.$

\section{Proof of Theorem \ref{main-triangles}}\label{section:triangles}

We will prove Dol'nikov's conjecture for triangles
(Theorem~\ref{main-triangles}) in the form of
Conjecture~\ref{conj:restated}. Since the statement does not
change under affine transformations of the plane, we restrict
ourselves to the case when the triangle $T$ is regular with unit
side length.

For the proof, we need the following two lemmas.

\begin{lemma}\label{lemma-triangles-1}
Let $T$ be a regular triangle in $\R^2$ with unit side length.
Assume that a finite set $X\subset R^2$ has width $h_1, h_2, h_3$
in the directions of the sides of $T$, and $h_1, h_2, h_3$ satisfy
the conditions below.
\begin{enumerate}
    \item[\rm 1.] $h_1\leq h_2\leq h_3$;
    \item[\rm 2.] $h_2\leq \sqrt{3}/2$;
    \item[\rm 3.] $h_1+h_2+h_3\leq 3\sqrt{3}/2$.
\end{enumerate}
Then $X$ can be covered by 3 translates of $T$.
\end{lemma}

\begin{lemma}\label{lemma-triangles-2}
Assume there are 2 finite sets $X_1, X_2 \subset \R^2$ and a line
$\ell$ in $\R^2$. Let $w_1$ and $w_2$ be the widths of $X_1$ and
$X_2$ respectively in the direction of $\ell$, and $w_1 + w_2 >
2w$. Then, there exist points $x_1\in X_1$, $x_2\in X_2$ such that
the width of the set $\{x_1, x_2\}$ in the direction of $\ell$ is
greater than $w$.
\end{lemma}

Before proving these lemmas, let us show how they can be used to prove Theorem~\ref{main-triangles}.

\begin{proof}[Proof of Theorem~\ref{main-triangles}]
Assume that none of the sets $X_1, X_2, X_3$ can be covered by 3 translates of $T$. Enumerate the sides of $T$ by 1, 2, 3 and denote by
$h_{ij}$ the width of $X_i$ in the direction orthogonal to the $j$-th side of $T$. Since it is impossible to cover $X_i$ by
3 translates of $T$, by lemma \ref{lemma-triangles-1} for each $i = 1,2,3$ at least one of the 2 conditions holds:
\begin{enumerate}
    \item At least 2 of the 3 widths $h_{i1}, h_{i2}, h_{i3}$ are greater than $\sqrt{3}/2$.
    \item $h_{i1} + h_{i2} + h_{i3} > 3\sqrt{3}/2$.
\end{enumerate}

Since there are 2 types of conditions, 2 of the 3 sets $X_1, X_2, X_3$ satisfy the condition of the same type. Assume that these sets are
$X_1$ and $X_2$.  Consider the following 2 cases.

\medskip
\noindent {\bf Case 1.} For $X_1$ and $X_2$ the first condition holds. Then $X_1$ and $X_2$ have a common direction such that their width in this direction
is greater than $\sqrt{3}/2$. So, without loss of generality assume $h_{11} > \sqrt{3}/2$ and $h_{21} > \sqrt{3}/2$.

\medskip
\noindent {\bf Case 2.} For $X_1$ and $X_2$ the second condition holds, i.e.
$$h_{11} + h_{12} + h_{13} > \frac{3\sqrt{3}}{2} \quad \text{and} \quad h_{21} + h_{22} + h_{23} > \frac{3\sqrt{3}}{2}.$$
Then for some index $j$ we have $h_{1j} + h_{2j} > \sqrt{3}$. Without loss of generality assume $h_{11} + h_{21} > \sqrt{3}$.

From both cases we concluded that $h_{11} + h_{21} > \sqrt{3}$. According to Lemma \ref{lemma-triangles-2}, there exist $x_1 \in X_1$ and $x_2 \in X_2$ such that the width of
$\{x_1, x_2\}$ in the direction of the first side of $T$ is greater than $\sqrt{3}/2$. Then $\rho(x_1, x_2) > 2$, which contradicts the assumption of
Conjecture~\ref{conj:restated}.

\end{proof}

The core of the proof is contained in Lemma \ref{lemma-triangles-1}, which we prove now.

\begin{proof}[Proof of Lemma \ref{lemma-triangles-1}]
Let $T_1$ be the minimal positive homothety of $T$ that contains $X$. Similarly, let $T_2$ be the minimal (by absolute value) negative
homothety of $T$ that contains $X$.

The intersection $T_1\cap T_2$ is a convex hexagon with all angles
equal to $2\pi/3$.  The only possible degenerate cases come from coincidences of vertices. Indeed, $T_2$ cuts from $T_1$ three regular triangles, one at each angle (possibly with a zero side length), and those triangles
are pairwise non-intersecting since each side of $T_1$ contains a point of $X$.

Denote the hexagon $T_1\cap T_2$ by $ABCDEF$, labelling its vertices in the cyclic order so that the segments $AB$, $CD$, $EF$ lie on the sides
of $T_1$ directed as sides 1, 2, 3 of $T$ respectively (see figure \ref{figure:definition-of-heights}).

\begin{figure}[h!]
\centerline{\includegraphics{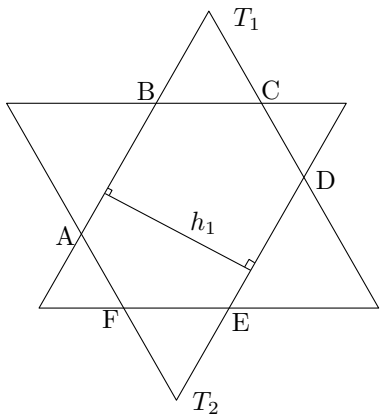}}
\caption{Construction of hexagon $ABCDEF$}\label{figure:definition-of-heights}
\end{figure}

 If $a$ is the side-length of $T_1$, then the side lengths of $ABCDEF$ are as follows.
\begin{eqnarray*}
AB = \frac{2}{\sqrt{3}}(h_2 + h_3) - a, & BC = a - \frac{2h_3}{\sqrt{3}} \\
CD = \frac{2}{\sqrt{3}}(h_1 + h_3) - a, & DE = a - \frac{2h_1}{\sqrt{3}} \\
 EF = \frac{2}{\sqrt{3}}(h_1 + h_2) - a, & FA = a - \frac{2h_2}{\sqrt{3}}.
\end{eqnarray*}

\begin{claim}\label{c1}
 $\max(AB, CD, EF) \leq 1$.
\end{claim}

\begin{proof} Indeed,
$$\max(AB, CD, EF) = \frac{2}{\sqrt{3}}(h_2 + h_3) - a = \frac{2h_2}{\sqrt{3}} - BC \leq \frac{2h_2}{\sqrt{3}} \leq 1. $$
\end{proof}

If $a\leq 1$, then the hexagon $ABCDEF$ is completely covered by one translate of $T$ (namely $T_1$). Since it contains the set $X$, in this case lemma is proved.
So we can assume $a>1$.

Now we want to choose points $K, L, M$ on $BC, DE, FA$ respectively so that

\begin{eqnarray*}
MA + AB + BK  & \leq & 1,\\
KC + CD + DL  & \leq & 1, \\
LE + EF + FM  & \leq &  1.
\end{eqnarray*}

Without loss of generality, assume that the order of points $A, B, C, D, E, F$ on the boundary of $ABCDEF$ is clockwise. We construct auxiliary points $M^* = A$ and $K^*$ and $L^*$ the points on the boundary of $ABCDEF$ such that the distance from $M^*$ to these points in the clockwise and counterclockwise
direction respectively along the boundary of $ABCDEF$ is equal to 1.  We will use these points to construct $K,L,M$.

First we prove that $L^*$ lies on $DE$. Indeed,
\[
M^*F + FE = AF + FE = 2h_1/\sqrt{3} \leq 1,
\]
and
 \[
 M^*F + FE + ED = AF + FE + ED = a > 1.
 \]

Observe that, according to the claim above, the arc $M^*K^*$ of $ABCDEF$ (the direction from $M^*$ to $K^*$ is clockwise) contains $B$. So there are 2 cases.

\medskip
{\bf Case 1.} $K^*\in BC$. Then we can take $K = K^*$, $L = L^*$, $M = M^*$. To check these points work, it is enough to verify the inequality
$$AB + BC + CD + DE + EF + FA \leq 3.$$
Indeed, since 2 arc lengths are equal to 1, the length of the remaining arc does not exceed 1. But
$$AB + BC + CD + DE + EF + FA = \frac{2}{\sqrt{3}}(h_1 + h_2 + h_3) \leq 3,$$
as we wanted.

\medskip
{\bf Case 2.} The clockwise directed arc $M^*K^*$ contains $C$. Then we start to move $M^*$ from $A$ to $F$ (together with $K^*$ and $L^*$) until one of the
following events happens.

\medskip
{\bf Case 2.1.} $L^*$ reaches $D$ before $K^*$ reaches $C$. Then set $K = C, L = D, M = M^*$. The length of the clockwise directed arc $MK$ does
not exceed 1, because $K^*$ did not reach $C$. The length of the counterclockwise directed arc $ML$ equals 1, and the length of the clockwise directed
arc $KL$ is at most 1 according to the claim above.

\medskip
{\bf Case 2.2.} $K^*$ reaches $C$ before $L^*$ reaches $D$. In this case we can take $K = K^*$, $L = L^*$, $M = M^*$. The proof is analogous to the one for case 1.

At least one of the subcases 2.1 or 2.2 will happen. Indeed,
$$FE + ED = 2h_2/\sqrt{3} \leq 1,$$
thus $L^*$ reaches $D$ before $M^*$ reaches $F$.

Now, given points $K, L, M$ construct 3 regular triangles $Q_1,
Q_2, Q_3$ positively homothetic to $T$ as shown in Figure
\ref{f1-triangles}.

\begin{figure}[h]

      \centerline{\includegraphics[width=0.5\textwidth]{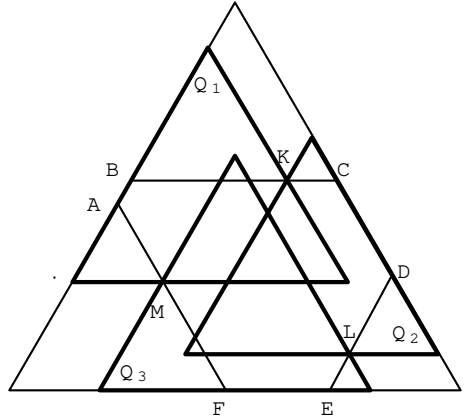}}
      \caption{Construction of triangles}\label{f1-triangles}

\end{figure}%

Namely, one of the triangles contains $AB$, $K$ and $M$ on the
boundary, the other 2 are constructed in a similar way.

Note that the side-length of $Q_1$ is equal to $MA + AB + BK$,
which is at most $1$ by construction. The same happens for $Q_2$
and $Q_3$. Moreover, if these triangles do not cover $ABCDEF$,
there is one of the 2 cases shown in figure \ref{f2-triangles}. 

The situation in Figure \ref{f2-triangles}a) will be called {\it counterclockwise
hole}, respectively, the situation in Figure~\ref{f2-triangles}b) will be called {\it clockwise hole}. 
We have chosen these terms (counterclockwise/clockwise hole) for the following reason: the counterclockwise (respectively, clockwise) hole increases 
its size if the triangles $Q_1$, $Q_2$ and $Q_3$ move counterclockwise (respectively, clockwise) at the same speed along the sides of the triangle $T_1$.

\begin{figure}[h]

      \centerline{\includegraphics[width=0.75\textwidth]{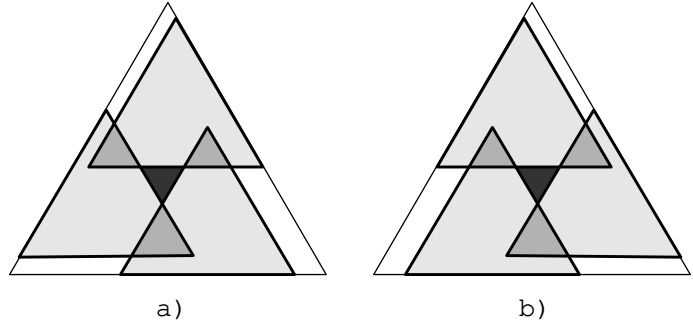}}
      \caption{Two types of holes}\label{f2-triangles}

\end{figure}%

Now we start to move the points $K$, $L$ and $M$ simultaneously at the same speed in the counterclockwise direction until the first of the following three events happens: $K$ reaches $B$, or $L$ reaches $D$, or $M$ reaches $F$. Without loss of generality, assume that $M$ = $F$.
 
Consider the construction of the triangles $Q_1$, $Q_2$ and $Q_3$ that uses the new positions of $K$, $L$ and $M$.
Since $M = F$, the lower sides of $Q_1$ and $Q_3$ lie on the lower side of $T_1$. Then a hole exists if and only if there is an intersection point of the 
right side of $Q_1$ with the left side of $Q_3$, and the lower side of $Q_2$ is higher then this intersection point.

Notice that such a hole is necessarily counterclockwise. Hence, if $M = F$ (or, similarly, $K = B$, or $L = D$), then a clockwise 
hole cannot exist.

Thereafter begin to move  the points $K$, $L$ and $M$ back, i.e., clockwise, at the same speed. We continue the motion until 
the first of the following events happens: $K$ reaches $C$, or $L$ reaches $E$, or $M$ reaches $A$. Similarly to the above, one can check that the 
triangles $Q_1$, $Q_2$ and $Q_3$ constructed for the new position of $K$, $L$ and $M$ cannot produce a counterclockwise hole.

We claim that while moving the points $K$, $L$ and $M$ we had such a position that the triangles $Q_1$, $Q_2$ and $Q_3$ constructed for 
this position produced no hole. If there was no hole at one of the ultimate positions of $K$, $L$ and $M$ (first --- with $M = F$, or $K = B$, or $L = D$, or
second --- with $M = A$, or $K = C$, or $L = E$), then there is nothing to prove. If there were holes at both ultimate positions, then, as we have proved, 
these were holes of different types. Hence, while moving $K$, $L$ and $M$ we had to switch the type of the hole. But two types of holes cannot exist
simultaneously, therefore for some position of $K$, $L$ and $M$ the triangles $Q_1$, $Q_2$ and $Q_3$ produced no hole at all.

But if the triangles $Q_1$, $Q_2$ and $Q_3$ produce no hole, then they cover the hexagon $ABCDEF$. Hence they cover the set $X$.
\end{proof}

\begin{proof}[Proof of Lemma \ref{lemma-triangles-2}]
Denote by $m$ a line perpendicular to $\ell$. Choose a positive
(right) direction on $m$. The orthogonal projection of $\conv X_i$
onto $m$ is an interval $I_i$ of length $w_i$.

Let $z_i$ be the midpoint of $I_i$ ($i = 1,2$). Without loss of generality, assume that $z_2$ lies on the right half-line with respect to $z_1$.
Denote by $y_1$ the left endpoint of $I_1$ and by $y_2$ the right endpoint of $I_2$. Then we have
$$|y_2 - y_1| = |y_2 - z_2| + |z_2 - z_1| + |z_1 - y_1| \geq \frac{w_1 + w_2}{2} > w.$$

Obviously, $y_1$ and $y_2$ have pre-images in $X_1$ and $X_2$ respectively under the orthogonal projection onto $\ell$. Denote these pre-images by
$x_1$ and $x_2$ respectively. The width of $\{x_1, x_2\}$ in the direction of $\ell$ is exactly $|y_2 - y_1|$, i.e. greater than $w$,
hence the statement of Lemma \ref{lemma-triangles-2}.

\end{proof}

\section{Proof of Theorem \ref{thm:four-colors}}\label{section:final-remark}

In order to show that Dol'nikov's conjecture holds if we use four
colours instead of three, we will need the following argument that
Karasev used in \cite{Kar08}.

First, we need the following construction.  Let $\mathcal F$ be a family of pairwise intersecting translates of a compact convex set $K$ in
the plane. For every triple $\{K_1, K_2, K_3\}\subset \mathcal F$
with $K_1\cap K_2\cap K_3=\emptyset,$ there is a triangle
$\triangle x_1x_2x_3$ with $x_1\in K_1,$ $x_2\in K_2,$ $x_3\in
K_3,$ and $\text{int}(\triangle x_1x_2x_3)\cap (K_1\cup K_2\cup
K_3)=\emptyset$, such that the line through $x_1$ and parallel to
$[x_2,x_3]$ is a supporting line of $K_1,$ and similarly for the
lines through $x_2,$ $x_3,$ which are parallel to $[x_1,x_3]$ and
$[x_1,x_2],$ respectively.  The following claim is proved in \cite{Kar08}.

\begin{claim}
Consider all triples
$\{K_1,K_2,K_3\}\subset\mathcal F$ and their respective triangles
$\triangle x_1x_2x_3$. Let $\triangle
x_1'x_2'x_3'$ be one of such triangles with maximal area. Then, any other member of
$\mathcal F$ contains one point of the set $\{x_1',x_2',x_3'\}.$
\end{claim}

\begin{proof}[Proof of Theorem \ref{thm:four-colors}]
Let $\mathcal F_1,$ $\mathcal F_2,$ $\mathcal F_3,$ $\mathcal
F_4,$ be four families of translates of $K$ such that any two
members of different family have non-empty intersection. Consider
all rainbow triples $\{K_1, K_2, K_3\}$ with empty intersection
and their respective triangles $\triangle x_1x_2x_3$.  Among
them, let $\triangle x_1'x_2'x_3'$ be a triangle with
maximal area, with $x_1'\in K_1,$ $x_2'\in K_2,$ $x_3'\in K_3$.
Without loss of generality suppose $K_1\in\mathcal F_1,$
$K_2\in\mathcal F_2$, $K_3\in\mathcal F_3$. Then, any member
$K_4\in\mathcal F_4$ contains one of the points $x_1',x_2',x_3'$.
\end{proof}

\section*{Acknowledgments}
The authors would like to thank Vladimir Dol'nikov for proposing
this problem and Imre B\'ar\'any for the support received at the
R\'enyi Institute, which made this collaboration possible.  We
would also like to thank Dol'nikov and B\'ar\'any for the helpful
discussions on the subject.  We appreciate the careful comments of the anonymous referee, which have greatly improved the quality of this paper.
\bibliographystyle{amsplain}

\bibliography{references}
\end{document}